\documentclass[11pt,letterpaper]{article}

\usepackage{graphicx}
\usepackage{enumerate}
\usepackage{amsmath}
\usepackage[english]{babel}
\usepackage{amssymb}
\usepackage{epsfig}
\usepackage{pstricks,pst-node,pst-coil,pst-plot,pst-text}
\usepackage{ifpdf}

\newtheorem{theorem}{Theorem}

\newtheorem{lemma}[theorem]{Lemma}
\newtheorem{proposition}[theorem]{Proposition}

\newlength{\saveparindent}
\def\bproof{\begin{rm}\protect\vspace{5pt}\noindent\textbf{Proof: }%
\addtolength{\parskip}{4pt}\setlength{\parindent}{0pt}}
\def\eproof{\end{rm}\addtolength{\parskip}{-4pt}%
\setlength{\parindent}{\saveparindent}}

\newcommand{\bprooff}[1]{\begin{rm}\protect\vspace{5pt}%
\noindent\textbf{Proof of #1: }\addtolength{\parskip}{4pt}%
\setlength{\parindent}{0pt}}

\newenvironment{proof}{\par\bproof}{\eproof\(\qed\) \par\medskip}
\newenvironment{prooff}[1]{\par\bprooff{#1}}{\eproof\(\qed\)\par}

\newcommand{\qed}{\quad\mbox{\rule{7pt}{7pt}}}

\newcommand{\NN}{\mathbb{N}}

\newcommand{\rmC}{\mathrm{C}}
\newcommand{\rmE}{\mathrm{E}}
\newcommand{\rmG}{\mathrm{G}}
\newcommand{\rmF}{\mathrm{F}}
\newcommand{\rmR}{\mathrm{R}}
\newcommand{\rmT}{\mathrm{T}}
\newcommand{\rmV}{\mathrm{V}}
\newcommand{\rmW}{\mathrm{W}}
\newcommand{\rmZ}{\mathrm{Z}}
\newcommand{\rmS}{\mathrm{S}}
\newcommand{\calR}{\mathcal{R}}
\newcommand{\calM}{\mathcal{M}}
\newcommand{\calW}{\mathcal{W}}
\newcommand{\calZ}{\mathcal{Z}}

\newcommand{\degen}[1]{\mathrm{s}(#1)}
\newcommand{\inner}[1]{\mathrm{I}(#1)}
\newcommand{\outter}[1]{\mathrm{O}(#1)}
\newcommand{\bottom}[1]{\lfloor #1\rfloor}
\newcommand{\topedge}[1]{\lceil #1\rceil}

\newcommand{\vv}{\textit{\textbf{v}}}
\newcommand{\one}{\textbf{\textsc{1}}}

\title{Satisfying states of triangulations of a convex $n$-gon }

\author{
  A.~Jim\'enez\thanks{
  Depto.~Ing.~Matem\'{a}tica, U.~Chile.
  Email: \texttt{ajimenez@dim.uchile.cl}.
  Gratefully acknowledges the support of Mecesup via UCH0607 Project,
    CONICYT via Basal-FONDAP in Applied Mathematics,
    FONDECYT 1090227 and the partial support of the Czech Research Grant MSM 0021620838 while
visiting KAM MFF UK.}
\and
  M.~Kiwi\thanks{
Depto.~Ing.~Matem\'{a}tica \&
  Ctr.~Modelamiento Matem\'atico UMI 2807, U.~Chile.
  Web: \texttt{www.dim.uchile.cl/$\sim$mkiwi}.
  Gratefully acknowledges the support of
    CONICYT via Basal-FONDAP in Applied Mathematics and
    FONDECYT 1090227.}
  \and
  M.~Loebl\thanks{Department of Applied Mathematics and 
  Institute for Theorical Computer Science, Charles University.
  Web: \texttt{kam.mff.cuni.cz/$\sim$loebl/}.
  Partially supported by Basal project Ctr.~Modelamiento Matem\'atico, U.~Chile.}
}

\begin{document}

\maketitle
\begin{abstract}
In this work we count the number of satisfying states of triangulations of a convex $n$-gon
using the transfer matrix method. 
We show an exponential (in $n$) lower bound. 
We also give the exact formula for the number of satisfying states of a
  strip of triangles.
\end{abstract}

\section{Introduction}\label{sec1} 
A classic theorem of Petersen claims that every cubic (each degree
 $3$) graph with no cutedge has a perfect matching. 
A well-known conjecture of Lovasz and Plummer from the mid-1970's,
 still open, asserts that for every cubic graph $\rmG$ with no
 cutedge, the number of perfect matchings of $\rmG$ is exponential in
 $|\rmV(\rmG)|$. 
The assertion of the conjecture was proved for the
 $k-$regular bipartite graphs by Schrijver~\cite{S}~ and for the
 planar graphs by Chudnovsky and Seymour~\cite{CS}. 
Both of these results are difficult.  
In general, the conjecture is widely open;
 see~\cite{KSS} for a linear lower bound obtained so far.

We suggest to study the conjecture of Lovasz and Plummer in the dual
setting. This relates the conjecture to a phenomenon well-known in
statistical physics, namely to the degeneracy of the Ising model on
totally frustrated triangulations of $2-$dimensional surfaces.

In order to explain this we need to start with another well-known
  conjecture, namely the directed cycle double cover conjecture of
  Jaeger (see~\cite{J}):
  \emph{Every cubic graph with no
  cutedge may be embedded to an orientable surface so that each face is
  homeomorphic to an open disc (i.e., the embedding defines a map) and
  the geometric dual has no loop.}

By a slight abuse of notation we say that a map in a $2-$dimensional 
  surface is a \emph{triangulation} if each face is bounded by a cycle 
  of length $3$ (in particular there is no loop); hence we allow multiple 
  edges.
We say that a set $\rmS$ of edges of a triangulation $\rmT$ is 
  \emph{intersecting} if $\rmS$ contains exactly one edge of 
  each face of $\rmT$. 

Assuming the directed cycle double cover conjecture, we can
 reformulate the conjecture of Lovasz and Plummer as follows: \emph{Each
 triangulation has an exponential number of intersecting sets of
 edges}.

We next consider the Ising model. 
Given a triangulation
  $\rmT=(\rmV,\rmE)$, we associate the \emph{coupling constant} 
  $c(e)=-1$ with each edge $e\in \rmE$. A \emph{state} of the Ising model is
  any function $\sigma: \rmV\rightarrow \{1,-1\}$. 
The energy of a state $s$
  is defined as $-\sum_{\{u,v\}\in \rmE}c(uv)\sigma(u)\sigma(v)$, and the states
  of minimum energy are called \emph{groundstates}.
The number of groundstates is usually called the \emph{degeneracy} of
  $\rmT$, denoted 
  $\mathrm{g}{\rmT}$,
  and it is an extensively studied quantity (for regular lattices
  $\rmT$) in statistical physics. 
Moreover, a basic tool in the
  degeneracy study is the transfer matrix method.

We further say that a state $\sigma$ \emph{frustrates} edge $\{u,v\}$ if $\sigma(u)=\sigma(v)$. Clearly, each state frustrates at least one edge of each
face of $\rmT$, and a state is a groundstate if it frustrates the smallest
possible number of edges. We say that a state $\sigma$ is \emph{satisfying}
if $\sigma$ frustrates exactly one edge of each face of $\rmT$.
Hence, the set of the frustrated edges of any satisfying state is 
an intersecting set defined above, and we observe:
The number of the satisfying states is at most twice the number of the 
  intersecting sets of edges.
Moreover, the converse also holds for planar triangulations: if we
  delete an intersecting set of edges from a planar triangulation, we
  get a bipartite graph and its bipartition defines a pair of satisfying
  states.

We finally note that a satisfying state does not need to exist, but if it exists,
then the set of the satisfying states is the same as the set of the groundstates.  

Summarizing, half the number of satisfying states is a lower bound to
  the number of intersecting sets. 
We can also formulate the result of Chudnovsky and Seymour by:
  \emph{Each planar triangulation has an exponential degeneracy}.
This motivates the problem we study as well as the (transfer matrix)
  method we use.

\medskip
Given $\rmC_{n}$ a convex $n$-gon, a \emph{triangulation of $\rmC_n$} is 
  a plane graph obtained from $\rmC_n$ by adding $n-3$ new edges so that 
  $\rmC_n$ is its boundary (boundary of its outer face).
We denote by $\Delta(\rmC_{n})$ the set of all triangulations of
  $\rmC_{n}$.
An \emph{almost-triangulation} is a plane graph so
  that all its inner faces are triangles.
Note that if $n \geq 3$, then $\Delta(\rmC_{n})$ is a subset of 
  the set of almost-triangulations with $n-2$ inner faces.
For $\rmT$ an almost-triangulation, we say that a state $\sigma$ is \emph{satisfying}
if $\sigma$ frustrates exactly one edge of each triangular face of $\rmT$. We denote by $\degen{\rmT}$ the number of satisfying states of an almost-triangulation
  $\rmT$.
The main goal of this work  
  is to show that the number of satisfying states of any triangulation of 
  a convex $n$-gon is exponential in $n$.  

\medskip
\noindent\textbf{Organization:}
We first 
  recall, in Section~\ref{sub1}, a known and simple bijection between 
  triangulations of a convex $n$-gon and plane ternary trees  
  with $n-2$ internal vertices. 
We then formally state the main results of this work.
In Section~\ref{sub2} we give a constructive 
  step by step procedure that given a plane ternary tree $\Gamma$
  with $n-2$ internal vertices, sequentially builds 
  a triangulation $\rmT$ of a convex $n$-gon
  by repeatedly applying one of three different elementary operations.
Finally, in Section~\ref{sec:degen} we interpret each elementary 
  operation in terms of operations on matrices.
Then, we apply the transfer matrix method to obtain, for each
  triangulation of a convex $n$-gon $\rmT$, an expression for a matrix
  whose coordinates add up to the number of satisfying states of 
  $\rmT$. 
We then derive a closed formula for the number of satisfying states of 
  a natural subclass of $\Delta(\rmC_{n})$; the class of ``triangle
  strips''.
Finally, we establish an exponential 
  lower bound for the number of satisfying states of triangulations of a convex
  $n$-gon.
Future research directions are discussed in Section~\ref{sec:conclusion}.

\section{Structure of the class of triangulations of a convex $n$-gon}\label{sub1}
Let $\rmT$ be a triangulation of a convex $n$-gon.
Denote by $\rmF(\rmT)$ the set of inner faces of $\rmT$ and 
  let  $\{\inner{\rmT}, \outter{\rmT}\}$ be the partition of 
  $\rmF(\rmT)$ such that $\Delta \in \inner{\rmT}$ 
  if and only if no edge 
  of $\Delta$ belongs to the boundary of $\rmT$ (i.e.~to $\rmC_{n}$). 
We henceforth refer to the elements of $\inner{\rmT}$ by
  \emph{interior triangles of $\rmT$}.
Consider now the bijection $\Gamma$ between $\Delta(\rmC_{n})$ and the set
  of all plane ternary trees with $n-2$ internal vertices and $n$ leaves
  that maps $\rmT$ to $\Gamma_{\rmT}$ so that:
\begin{itemize}
\item[(i)] $\{\gamma_{\Delta},\gamma_{\Delta'}\}$ is an edge of $\Gamma_{\rmT}$
  if and only if $\Delta$ and $\Delta'$ are inner 
  faces of $\rmT$ that share an edge, and

\item[(ii)] $e$ is a leaf of $\Gamma_{\rmT}$ adjacent to $\gamma_{\Delta}$
  if and only if $e$ is an edge of $\rmC_n$ that belongs to $\Delta$. 
\end{itemize}
(See Figure~\ref{fig:bijection} for an illustration of how $\Gamma$ acts on an
  element of $\Delta(\rmC_n)$.) 
The bijection $\Gamma$ induces another bijection, say $\gamma$, 
  from the inner faces of $\rmT$ (i.e.~$\rmF(\rmT)$), to
  the internal vertices of~$\Gamma_{\rmT}$.
In particular, inner faces $\Delta$ and $\Delta'$ of $\rmT$ 
  share an edge if and only if 
  $\{\gamma_{\Delta},\gamma_{\Delta'}\}$ is an edge of $\Gamma_{\rmT}$
  which is not incident to a leaf.
Hence, $\gamma$ identifies interior triangles of $\rmT$ with 
  internal vertices of $\Gamma_{\rmT}$ that are not adjacent
  to leaves.
\begin{figure}[h]
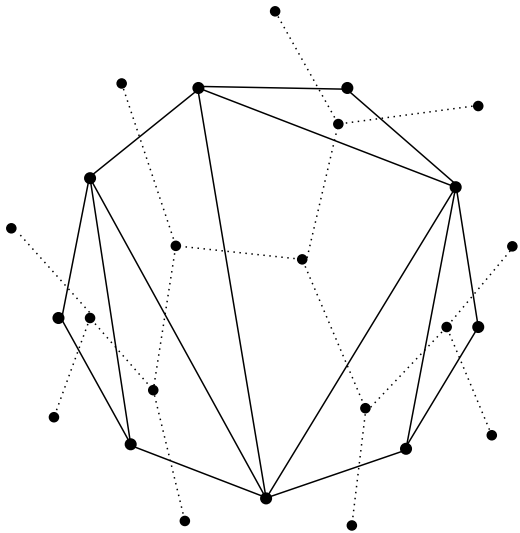

\centering
\ifpdf\input{figure1.pdf_t}\else\input{figure1.pstex_t}\fi
\caption{A triangulation of a convex $9$-gon $\rmT$ and the associated tree $\Gamma_{\rmT}$.}\label{fig:bijection}
\end{figure}

\subsection{Main results}
Say a triangulation of a convex $n$-gon $\rmT$ is a 
  \emph{strip of triangles} provided $|\inner{\rmT}|=0$.
Our first result is an exact formula for the number of satisfying states
  of any strip of triangles.
Our second main contribution gives an exponential lower bound for the
  number of satisfying states of any triangulation of a convex $n$-gon.
Specifically, denoting by $F_{k}$ the $k$-th Fibonacci number and 
  $\varphi=(1+\sqrt{5})/2 \approx 1.61803$ the golden ratio,
we establish the following results:
\begin{theorem}\label{result1} 
If $\rmT$ is a triangulation of a convex $n$-gon such that 
  $|\inner{\rmT}|=0$, then $\degen{\rmT} = 2 F_{n+1}$.
\end{theorem}
\begin{theorem}\label{result2}
If $\rmT$ is a triangulation of a convex $n$-gon, then 
  $\degen{\rmT}\geq \varphi^{2}(\sqrt{\varphi})^{n}$.
Moreover, $\sqrt{\varphi} \approx 1.27202$.
\end{theorem}

\section{Construction of triangulations of a convex $n$-gon}\label{sub2}
In this section we discuss how to iteratively construct any triangulation
  of a convex $n$-gon.
First, we introduce two basic operations whose repeated application
  allows one to build strips of triangles.
Then, we describe a third operation which is crucial for
  recursively building triangulations 
  with a non-empty set of interior triangles 
  from triangulations with fewer interior triangles.

\subsection{Basic operations}
Let $\rmT=(\rmV,\rmE)$ be a triangulation of a convex $n$-gon.
We will often distinguish a boundary edge of $\rmT$ to which we shall
  refer as \emph{bottom edge of $\rmT$} and denote by 
  $\bottom{\rmT}$.

We now define two elementary operations (see Figure~\ref{fig:top} for an 
illustration):
\begin{quote}
\begin{tabular}{lp{5in}}
\multicolumn{2}{l}{\textbf{Operation $\rmW$}} \\
Input: & $(\rmT,\bottom{\rmT})$ where
  $\rmT\in\Delta(\rmC_n)$ and $\bottom{\rmT}=(\beta_1,\beta_2)$. \\
Output: & $(\widehat{\rmT},\bottom{\widehat{\rmT}})$, 
  where $\widehat{\rmT}\in\Delta(\rmC_{n+1})$ 
  is a triangulation obtained from $\rmT$ by 
  adding a new vertex $\widehat{\beta}_{1}$ to $\rmT$ and  
  two new edges $\{\widehat{\beta}_{1}, \beta_1\}$ and 
  $\{\widehat{\beta}_{1}, \beta_2\}$. 
Moreover, $\bottom{\widehat{\rmT}}$ =
  $(\widehat{\beta}_{1},\beta_2)$.
\end{tabular}

\begin{tabular}{lp{5in}}
\multicolumn{2}{l}{\textbf{Operation $\rmZ$}} \\
Input: & $(\rmT,\bottom{\rmT})$ where 
  $\rmT\in\Delta(\rmC_n)$ and $\bottom{\rmT}=(\beta_1,\beta_2)$. \\
Output: &  $(\widehat{\rmT},\bottom{\widehat{\rmT}})$, 
  where $\widehat{\rmT}\in\Delta(\rmC_{n+1})$ 
  is a triangulation obtained from
  $\rmT$ by adding a new vertex $\widehat{\beta}_{2}$ to $\rmT$ and
  two new edges $\{\beta_1,\widehat{\beta}_{2}\}$ and 
  $\{\widehat{\beta}_{2}, \beta_2\}$. 
Moreover, $\bottom{\widehat{\rmT}} = (\beta_1,\widehat{\beta}_{2})$.
\end{tabular}
\end{quote}
Henceforth, we also view operations $\rmW$ and $\rmZ$ as 
  maps from inputs to outputs.
Abusing terminology, we consider two nodes joined by an edge 
  to be a \emph{degenerate triangulation} whose bottom edge is its
  unique edge.
Let $\rmT_0$ be a degenerate triangulation. Say that $\bottom{\rmT_0}$ is the \emph{top edge of $\rmT$},
  denoted $\topedge{\rmT}$, if
  there is a sequence $\rmR_1,\ldots,\rmR_l\in\{\rmW,\rmZ\}$
  such that $(\rmT,\bottom{\rmT})$ is obtained by evaluating 
  $\rmR_l\circ\cdots\circ\rmR_2\circ\rmR_1$ at 
  $(\rmT_0,\bottom{\rmT_0})$.
When bottom edges are clear from context, we shall simply write
\begin{equation*}
\rmT = \rmR_l\circ\cdots\circ\rmR_2\circ\rmR_1(\rmT_0)\,.
\end{equation*}

\begin{figure}[h]
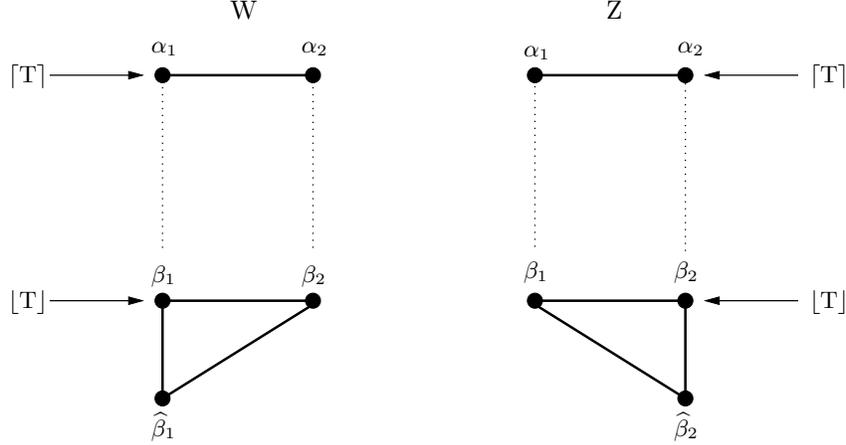

\centering
\ifpdf\input{figure2.pdf_t}\else\input{figure2.pstex_t}\fi\caption{An arbitrary strip of triangles $\rmT$. Operations $\rmW$ and $\rmZ$ evaluated at $(\rmT, \bottom{\rmT})$.}\label{fig:top}
\end{figure}

\subsection{The $|\inner{\rmT}| = 0$ case}
Our goal in this section is to show that 
  any triangulation of a convex $n$-gon with no interior
  triangles can be obtained by sequentially applying basic operations 
  of type $\rmW$ and $\rmZ$ starting from a degenerate triangulation.

Let $\rmT$ be a triangulation such that $|\inner{\rmT}|=0$.
Note that each internal vertex of $\Gamma_{\rmT}$ is adjacent to at 
  least one leaf. 
Hence, $\Gamma_{\rmT}$ has two internal vertices each one adjacent
  to exactly two leaves, and $n-4$ internal vertices adjacent to 
  exactly one leaf. 
This implies that $\Gamma_{\rmT}$ is made up of a
  path $P = \gamma_{\Delta^1} \ldots \gamma_{\Delta^{n-2}}$ with
  two leaves connected to each $\gamma_{\Delta^1}$ and
  $\gamma_{\Delta^{n-2}}$, and one leaf connected to each internal
  vertex of the path $P$ (see Figure~\ref{fig:path}). 
To obtain $\rmT$ from $\Gamma_{\rmT}$ we choose one of the 
  two endnodes of the path (say $\gamma_{\Delta^1}$) 
  and sequentially add the triangles $\Delta^1,\ldots,\Delta^{n-2}$ 
  one by one, according to the bijection $\gamma$, starting
  from $\gamma_{\Delta^1}$ and following the trajectory of the 
  path~$P$.
Consequently, we can construct $\rmT$ from a pair of vertices 
  $(\alpha_1,\alpha_2)$ of $\Delta^1$ by applying a sequence 
  of $n-2$ operations $\rmR_{1}, \rmR_{2}, \ldots,
  \rmR_{n-2}\in\{\rmW,\rmZ\}$, where the
  choice of each operation depends on the structure of
  $\Gamma_{\rmT}$. 
For example, for the triangulation in Figure~\ref{fig:path}, 
  provided $\topedge{\rmT}=(\alpha_{1},\alpha_{2})$ 
  and $\bottom{\rmT}=(\beta_{1},\beta_{2})$, we have that 
  $\rmR_1=\rmW$, $\rmR_2=\rmZ$, $\rmR_3=\rmZ$, and so on and so forth.

\begin{figure}[h]
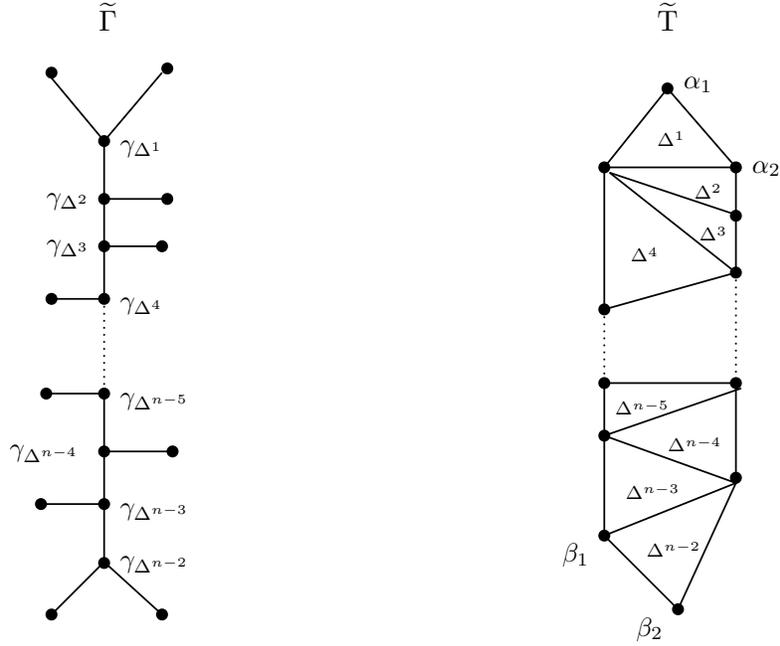

\centering
\ifpdf\input{figura3.pdf_t}\else\input{figura3.pstex_t}\fi
\caption{A tree $\widetilde{\Gamma}$ in the range of bijection
  $\Gamma$ and construction of triangulation $\widetilde{\rmT}$ 
  such that $\Gamma_{\widetilde{\rmT}}=\widetilde{\Gamma}$.}\label{fig:path}
\end{figure}

The next result summarizes the conclusion of the previous discussion.
\begin{lemma}\label{strip} 
For any $\rmT \in \Delta(\rmC_{n})$ it holds that 
 $|\inner{\rmT}|=0$ 
  if and only if there is a degenerate triangulation $\rmT_0$
  and basic operations 
  $\rmR_{1}, \rmR_{2}, \ldots,\rmR_{n-2}\in\{\rmW,\rmZ\}$ 
  such that
\begin{equation*} 
\rmT
  =\rmR_{n-2}\circ\cdots\circ\rmR_{2}\circ\rmR_{1}(\rmT_0)\,.
\end{equation*}
In fact, there are non--negative integers  
  $w_1, \ldots, w_{m}, z_1, \ldots, z_{m}$ adding up to $n-2$
  such that $w_j \geq 1$ for $j\neq 1$,
  $z_j \geq 1$ for $j\neq m$, and 
\begin{equation*}
\rmT=
  \rmZ^{z_m}\circ\rmW^{w_m}\circ
  \cdots\circ\rmZ^{z_{2}}\circ\rmW^{w_{2}}\circ\rmZ^{z_1}\circ\rmW^{w_1}
  (\rmT_0)\,.
\end{equation*}
\end{lemma}

\subsection{The $|\inner{\rmT}| \geq 1$ case}
We now consider the following additional basic operation
  (see Figure~\ref{fig:bullet} for an illustration):
\begin{quote}
\begin{tabular}{lp{5in}}
\multicolumn{2}{l}{\textbf{Operation $\bullet$}} \\
Input: & $(\rmT_i,\bottom{\rmT_i})$ where 
  $\rmT_i\in\Delta(\rmC_{n_i})$, $i\in\{1,2\}$
  and $\bottom{\rmT_i}=(\beta^i_1,\beta^i_2)$. \\
Output: &  $(\rmT,\bottom{\rmT})$, 
  where $\rmT\in\Delta(\rmC_{n_1+n_2-1})$ 
  is a triangulation obtained from
  $\rmT_1$ and $\rmT_2$ by identifying $\beta_1^2$ with 
  $\beta_2^1$ and adding the edge $\{\beta_1^1,\beta_2^2\}$.
Moreover, $\bottom{\rmT} = (\beta_{1}^1,\beta_{2}^2)$.
\end{tabular}
\end{quote}

\begin{figure}[h]
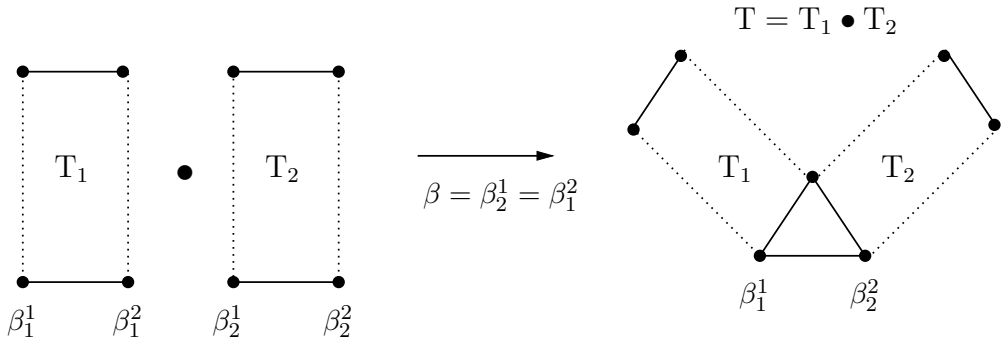

\centering
\ifpdf\input{figura4.pdf_t}\else\input{figura4.pstex_t}\fi\caption{Building an interior triangle by means of operation $\bullet$.}\label{fig:bullet}
\end{figure}

Assume $\rmT$ is such that 
  $|\inner{\widetilde{\rmT}}|=1$.
In particular, let $\inner{\rmT} = \{\Delta\}$. 
Clearly, the tree $\Gamma_{\rmT}$ contains exactly one internal 
  vertex that is not adjacent to a leaf.
Hence, in $\Gamma_{\rmT}$ there must be three internal vertices each
  of them adjacent to two leaves, and $n-6$ internal vertices 
  adjacent to exactly one leaf. 
Thus, we can identify in $\Gamma_{\rmT}$ three paths
  $P_1 = \gamma_{\Delta_{1}^1} \ldots\gamma_{{\Delta_{n_1}^1}}$,
  $P_2 = \gamma_{\Delta_{1}^2} \ldots\gamma_{{\Delta_{n_2}^2}}$,
  and $P_3 = \gamma_{\Delta_{n_3}^3}\ldots\gamma_{{\Delta_{1}^3}}$
  with end-vertices 
  $\gamma_{{\Delta_{n_1}^1}} = \gamma_{{\Delta_{n_2}^2}} =
    \gamma_{{\Delta_{n_3}^3}}=\gamma_{\Delta}$,
  and such that: 
  (1) $n_1 + n_2 + n_3 = n$ and $n_1, n_2, n_3\geq 2$,
  (2) each $\gamma_{{\Delta_{1}^j}}$ with $j \in \{1,2,3\}$ 
  is adjacent to two leaves of $\Gamma_\rmT$, and
  (3) each $\gamma_{{\Delta_{i_j}^j}}$ with 
  $j \in \{1,2,3\}$ and $i_j \in \{2, \ldots, n_j-1\}$ is adjacent
  to a single leaf of $\Gamma_{\rmT}$.

Given $\Gamma_\rmT$, we can construct $\rmT$
  by means of the following iterative step by step procedure:

\begin{enumerate}
\item For $i\in\{1,2\}$, add triangles 
  $\Delta_{1}^i,\ldots,\Delta_{n_i-1}^i$
  according to the bijection following the trajectory 
  from $\gamma_{\Delta_{1}^i}$ to $\gamma_{\Delta_{n_i-1}^i}$ 
  given by $P_i$, thus obtaining a triangulation $\rmT_i$
  such that $\Gamma_{\rmT_i}$ is the minimal subtree of 
  $\Gamma_{\rmT}$ containing $P_i\setminus\gamma_{\Delta}$.
Moreover, note that $\rmT_i\in\Delta(\rmC_{n_i+1})$
  is such that $|\inner{\rmT_i}| = 0$, and that there is 
  a degenerate triangulation $\rmT_{i,0}$ which is an edge 
  of triangle $\Delta_1^i$, and basic 
  operations $\rmR_{1}^i,\ldots,\rmR_{n_i-1}^i\in\{\rmW,\rmZ\}$ such that
\begin{equation*}
\rmT_i = \rmR_{n_i-1}^i\circ\ldots\circ\rmR_{2}^i\circ\rmR_{1}^i(\rmT_{i,0})\,.
\end{equation*}
Also, note that $\bottom{\rmT_i}$ is an edge of $\Delta^i_{n_i-1}$.

\begin{figure}[h]
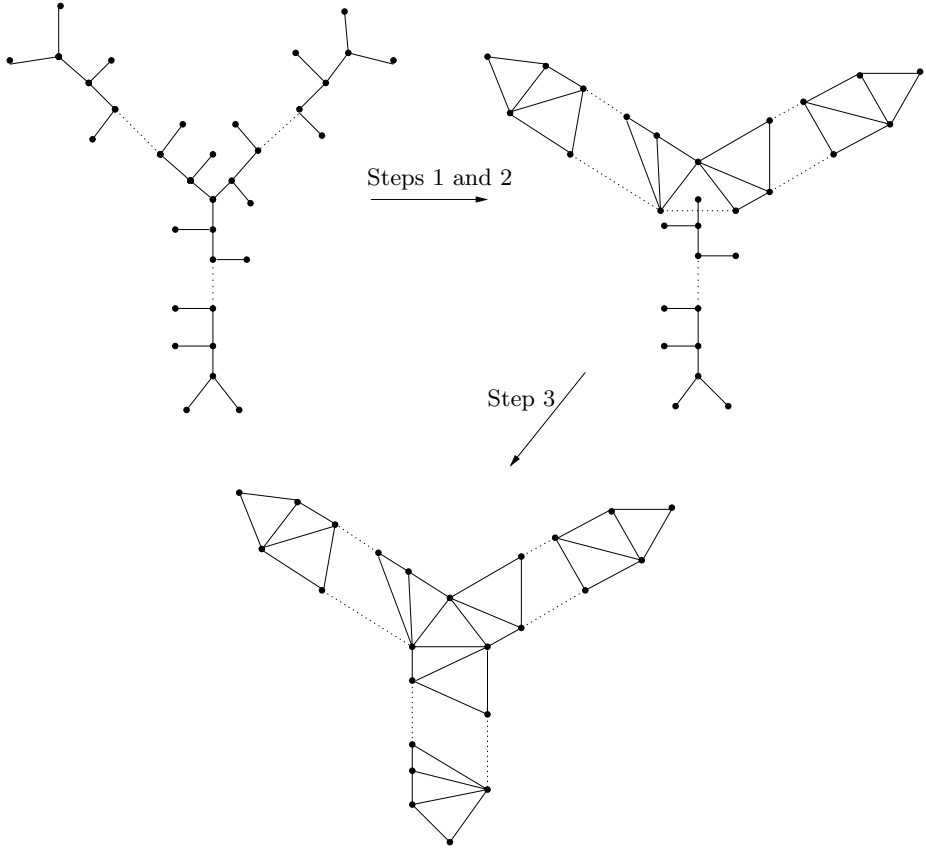

\centering
\ifpdf\input{figure5.pdf_t}\else\input{figure5.pstex_t}\fi\caption{Sketch of construction of an arbitrary $\rmT$ with $|\inner{\rmT}|=1$. }\label{fig:iterProcedure}
\end{figure}

\item Apply operation $\bullet$ in order to construct
  $\widehat{\rmT}=\rmT_1\bullet\rmT_2\in\Delta(\rmC_{n_1+n_2+1})$. 
  Note that $\Delta\in\rmF(\widehat{\rmT})$ and $\lfloor\widehat{\rmT}\rfloor$
  is the unique edge of $\Delta$
  which is in the boundary of $\widehat{\rmT}$.

\item Finally, starting from $\widehat{\rmT}$ add
   triangles associated to vertices of the path $P_3$. 
This is done by performing a sequence of $n_3-1$ 
  operations $\rmW$ and $\rmZ$ along $P_3\setminus\gamma_{\Delta}$
  starting from $(\widehat{\rmT},\bottom{\widehat{\rmT}})$. 
Given that $\widehat{\rmT}\in\Delta(\rmC_{n_1+n_2+1})$, 
  we obtain $\rmT\in\Delta(\rmC_{n_1+n_2+n_3})$
  (recall that $n_1+n_2+n_3=n$).
\end{enumerate}

We summarize the previous discussion as follows:
\begin{lemma}\label{lemma2}
Let $\rmT$ be a triangulation of a convex $n$-gon
  such that $|\inner{\rmT}|=1$. 
For some $n_1,n_2,n_3\geq 2$ such that 
  $n_1 + n_2 + n_3=n$, there are triangulations 
  $\rmT_1$ and $\rmT_2$ of convex $(n_1+1)$ and $(n_2+1)$-gons
  such that $|\inner{\rmT_1}|=|\inner{\rmT_2}|=0$, 
  and basic operations $\rmR_1,\ldots,\rmR_{n_3-1}\in\{\rmW,\rmZ\}$
  such that 
\begin{equation*}
\rmT =  \rmR_{{n_3-1}}\circ\cdots\circ\rmR_{{2}}\circ
  \rmR_{{1}}(\rmT_1 \bullet \rmT_2)\,.
\end{equation*}
\end{lemma}

Now, we state the main result concerning the recursive 
  construction of an arbitrary triangulation of a convex 
  $n$-gon that we will need.
\begin{lemma}\label{mainlemma}
Let $\rmT$ be a triangulation of a convex $n$-gon
  such that $|\inner{\rmT}| = m \geq 2$. 
Then, there are $\widehat{n} \geq 5$, $\widetilde{n}\geq 3$ and $l\geq 1$ such 
  that $\widetilde{n}+\widehat{n}+l-1=n$, and 
  triangulations $\widetilde{\rmT}\in\Delta(\rmC_{\widetilde{n}})$ and 
  $\widehat{\rmT}\in\Delta(\rmC_{\widehat{n}})$ satisfying:
\begin{enumerate}
\item\label{it:lemone}
$|\inner{\widetilde{\rmT}}|=0$, 

\item\label{it:lemtwo} $(\widehat{\rmT}, \bottom{\widehat{\rmT}})$ is either:
\begin{enumerate}
\item\label{it:lemtwo.one}
The output of operation $\rmW$ or 
  $\rmZ$ and $|\inner{\widehat{\rmT}}|=m-1$, or
\item\label{it:lemtwo.two}
The output of operation 
  $\bullet$ and $|\inner{\widehat{\rmT}}|=m-2$.
\end{enumerate}

\item There are basic 
  operations $\rmR_1,\ldots,\rmR_{l}\in\{\rmW,\rmZ\}$
  for which $\rmT = \rmR_{l}\circ\cdots\circ\rmR_{2}\circ\rmR_{1}
  (\widetilde{\rmT}\bullet\widehat{\rmT})$.
\end{enumerate}
\end{lemma}
\begin{proof}
Observe that there must be an internal vertex of $\Gamma_{\rmT}$,
  say $\gamma_{\Delta}$, such that if
  $\Gamma_{\widehat{\rmT}}$, $\Gamma_{\widetilde{\rmT}}$ and 
  $\Gamma_{\rmT_{l+2}}$ are the three sub-trees of $\Gamma_{\rmT}$
  rooted in $\gamma_{\Delta}$, then
  all internal vertices of $\Gamma_{\widetilde{\rmT}}\setminus\gamma_{\Delta}$
  and $\Gamma_{\rmT_{l + 2}}\setminus\gamma_{\Delta}$
  are adjacent to at least one leaf.
In particular, $|\inner{\widetilde{\rmT}}|=|\inner{\rmT_{l+2}}|=0$, 
  and condition~\ref{it:lemone}
  of the statement of the lemma is satisfied.

Let $\gamma_{\widehat{{\Delta}}}$ be the neighbor of $\gamma_{\Delta}$ 
  in $\Gamma_{\widehat{\rmT}}$.
Note that one of the following two situations must occur:
\renewcommand{\labelenumi}{}
\begin{enumerate}
\item 
\textbf{Case \arabic{enumi}:} 
In $\Gamma_{\widehat{\rmT}}\setminus \gamma_{\Delta}$,
  the vertex $\gamma_{\widehat{\Delta}}$ is adjacent to a leaf 
  (see Figure~\ref{fig:claim}.(a)).
In particular,
  $\Gamma_{\widehat{\rmT}}$ has exactly $m-1$ internal vertices 
  which are not adjacent to any leaf, or

\item 
\textbf{Case \arabic{enumi}:} 
None of the neighbors of $\gamma_{\widehat{\Delta}}$ in
  $\Gamma_{\widehat{\rmT}}\setminus\gamma_{\Delta}$ are adjacent to leaves 
  (see Figure~\ref{fig:claim}.(b)).
In particular, 
  $\Gamma_{\widehat{\rmT}}$ has exactly $m-2$ internal vertices 
  which are not adjacent to any leaf.
\end{enumerate}
\renewcommand{\labelenumi}{enumi}%

\begin{figure}[h]
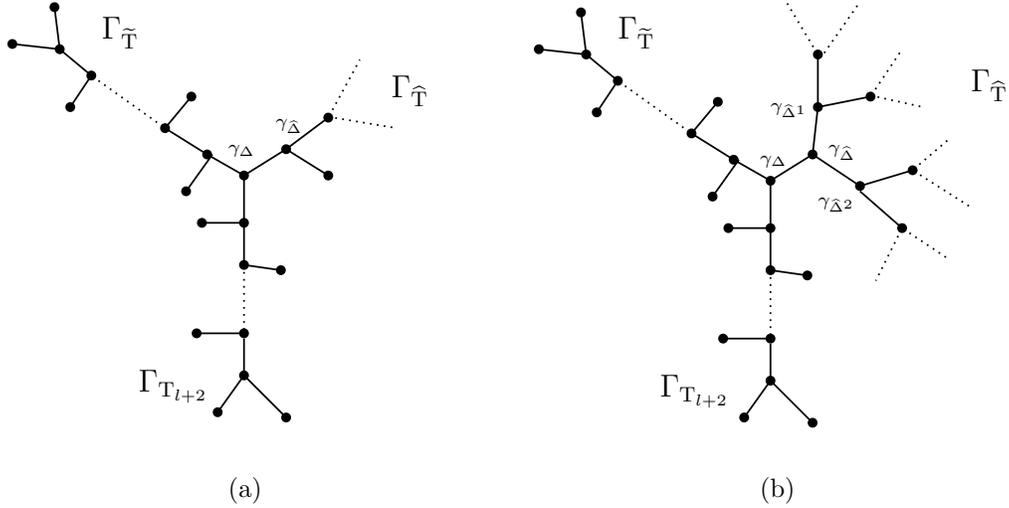
 
\centering
\ifpdf\input{figure6.pdf_t}\else\input{figure6.pstex_t}\fi\caption{Structure of $\Gamma_{\rmT}$ depending on 
  the one of subtree $\Gamma_{\widehat{\rmT}}$.}\label{fig:claim}
\end{figure}

Assume that the first case holds.
Recall that $|\inner{\widehat{\rmT}}|=m-1$.
Let $\widehat{\rmT}_0$ be the triangulation such that 
  $\Gamma_{\widehat{\rmT}_0}$ is the ternary tree 
  obtained from $\Gamma_{\widehat{\rmT}}\setminus\gamma_{\Delta}$ 
  by deleting the neighbor of $\gamma_{\widehat{\Delta}}$ which is a leaf. 
Let $\bottom{\widehat{\rmT}_0}$ be the edge of $\widehat{\rmT}_0$ 
  corresponding to the unique edge incident to 
  $\gamma_{\widehat{\Delta}}$ in $\Gamma_{\widehat{\rmT}_0}$.
Note that applying one basic operation 
  of type $\rmW$ or $\rmZ$ we can obtain 
  $(\widehat{\rmT},\bottom{\widehat{\rmT}})$ from 
  $(\widehat{\rmT}_0,\bottom{\widehat{\rmT}_0})$.
Therefore, $(\widehat{\rmT},\bottom{\widehat{\rmT}})$ 
  satisfies condition~\ref{it:lemtwo.one} of the statement of the 
  lemma.

Suppose now that the second case holds.
Recall that $|\inner{\widehat{\rmT}}| = m-2$.
Let $\gamma_{\widehat{\Delta}^1}$ and 
  $\gamma_{\widehat{\Delta}^2}$ be the vertices in 
  $\Gamma_{\widehat{\rmT}}\setminus\gamma_{\Delta}$
  that are neighbors of $\gamma_{\widehat{\Delta}}$.
Let $\Gamma_{\widehat{\rmT},1}$ 
  and $\Gamma_{\widehat{\rmT},2}$  be the trees obtained
  from $\Gamma_{\widehat{\rmT}\setminus\gamma_{\Delta}}$
  by removing the trees rooted at $\gamma_{\widehat{\Delta}^2}$
  and $\gamma_{\widehat{\Delta}^1}$, respectively.
Consider $i\in\{1,2\}$ and note that 
  $\Gamma_{\widehat{\rmT},i}$ is a ternary tree since by 
  hypothesis neither $\gamma_{\widehat{\Delta}^1}$ nor
  $\gamma_{\widehat{\Delta}^2}$ 
  are adjacent to leaves of $\Gamma_{\widehat{\rmT}}\setminus\gamma_{\Delta}$.
Let $\widehat{\rmT}_i$ be the triangulation that is 
  in bijective correspondence with $\Gamma_{\widehat{\rmT},i}$.
Define $\bottom{\widehat{\rmT}_i}$ to be the edge of 
  triangulation $\widehat{\rmT}_i$ which is in bijection with 
  the edge $(\gamma_{\widehat{\Delta}},\gamma_{\widehat{\Delta}^i})$
  of $\Gamma_{\widehat{\rmT},i}$. 
Note that $(\widehat{\rmT},\bottom{\widehat{\rmT}})$ may be obtained as
  $\widehat{\rmT}_1\bullet\widehat{\rmT}_2$.
Therefore, $(\widehat{\rmT},\bottom{\widehat{\rmT}})$ satisfies
   condition~\ref{it:lemtwo.two} of the statement of the lemma.

To finish the construction of $\rmT$ 
  it suffices to apply an appropriate sequence of $l$ operations from
  the set $\{\rmW,\rmZ\}$ starting from 
  $(\widetilde{\rmT}\bullet\widehat{\rmT},\bottom{\widetilde{\rmT}\bullet\widehat{\rmT}})$.
The result follows.
\end{proof}

\section{Satisfying States}\label{sec:degen}
In this section we first present a
  technique, the so called Transfer Matrix Method.
The technique is usually applied in situations where there
  is an underlying regular lattice, and gives formulas
  for its degeneracy.  
We adapt the technique to the context where instead of 
  a lattice there is a triangulation of 
  of a convex $n$-gon $\rmT$ and use it to determine $\degen{\rmT}$. 
Then, we apply the method to derive an exact formula for the
  number of satisfying states of strips of triangles.
Finally, we extend our arguments in order to establish 
  an exponential lower bound for $\degen{\rmT}$ 
  of any $\rmT$ triangulation of a convex $n$-gon.

\subsection{Transfer matrices and satisfying matrix}
Henceforth, the index of 
  rows and columns of all $4 \times 4$ matrices we consider 
  will be assumed to belong to $\{\texttt{+},\texttt{-}\}^{2}$.  
Let $\rmT$ be a triangulation of a convex $n$-gon such that
  $|\inner{\rmT}|=0$.
{From} now on, let $\one$ denote the $4\times 1$ vector all 
  of whose coordinates are $1$, i.e.~$\one = (1,1, 1, 1)^{t}$. 
Our immediate goal is to obtain a matrix 
  $\calM = \calM(\rmT)$ of type $4 \times 4$ that 
  satisfies the following two conditions:
\renewcommand{\labelenumi}{}
\begin{enumerate}
\item
  \textbf{Condition \arabic{enumi}:} 
Columns and rows of $\calM$ are indexed 
  by spin-assignments of the top and bottom node pairs of $\rmT$,
  respectively.
\label{it:cond1}

\item
  \textbf{Condition \arabic{enumi}:} 
For $\phi, \psi  \in \{\texttt{+},\texttt{-}\}^2$, 
  the value $\calM[\phi,\psi]$ is equal to the number of 
  satisfying states of $\rmT$ if the spin-assignments of
  the top and bottom node pairs of $\rmT$ are 
  $\psi$ and $\phi$, respectively.
\label{it:cond2}
\end{enumerate}
\renewcommand{\labelenumi}{enumi}%
Matrix $\calM$ is called the \emph{satisfying matrix} of $\rmT$.
It immediately follows that
\begin{equation*}
\degen{\rmT} = \one^{t}\cdot\calM\cdot\one\,.
\end{equation*}

By Lemma~\ref{strip}, each triangulation $\rmT \in \Delta(\rmC_{n})$ 
  such that $|\inner{\rmT}|=0$ may be constructed by applying a 
  sequence of $n-2$ operations of type $\rmW$ or $\rmZ$
  starting from $\rmT$'s top edge.
To each operation $\rmR\in\{\rmW,\rmZ\}$ we associate
  a so called \emph{transfer matrix} of type $4 \times 4$, 
  say $\calR\in\{\calW,\calZ\}$ such that:
\begin{itemize}
\item Columns of $\calR$ are indexed by 
  spin-assignments of the bottom node pair
  of $\rmT$.

\item Rows are indexed by spin-assignments of the bottom node pair
  of $\rmR(\rmT)$.

\item For $\phi, \psi  \in \{\texttt{+},\texttt{-}\}^2$, 
  matrix $\calR$  satisfies 
\begin{equation*}
\calR[\phi,\psi] = \left\{\begin{array}{cl}
  1\,, & \parbox[t]{4in}{if by setting the spin-assignments of 
  the bottom node pairs of $\rmT$ and $\rmR(\rmT)$ to $\psi$ 
  and $\phi$ respectively,  the state of the triangle created by 
  the application of $\rmR$ is satisfying,} \\
  0\,, & \text{otherwise.} 
\end{array}\right.
\end{equation*}
\end{itemize}

\begin{proposition}\label{transfer} 
Let $n\geq 3$ and $\rmT_0$ be a degenerate triangulation. 
Let $\rmT\in\Delta(\rmC_{n})$ 
  be such that 
  $\rmT = \rmR_{n-2}\circ \cdots \rmR_2\circ\rmR_1(\rmT_{0})$.
If $\calR_i$ denotes the transfer matrix associated to 
  $\rmR_i\in\{\rmW,\rmZ\}$, then
  $\calM(\rmT) = \calR_{n-2} \cdots \calR_2 \cdot \calR_1$.
\end{proposition}
\begin{proof}
We proceed by induction on $n$.
If $n=3$ we have that $\rmT = \rmR_1(\rmT_{0})$  
  and the statement follows by
  definition of $\calM(\rmT)$ and $\calR$.
Assume $n>3$.
By inductive hypothesis the satisfying matrix of the triangulation 
  $\widehat{\rmT} =\rmR_{n-3}\circ\cdots\circ\rmR_2\circ\rmR_1(\rmT_{0})
     \in \Delta(\rmC_{n-1})$  is
\begin{equation*}
\calM(\widehat{\rmT})= \calR_{n-3}\cdot\calR_{n-4} \cdots 
  \calR_2 \cdot \calR_1\,.
\end{equation*}
The matrix $\calR_{n-2} \cdot \calM(\widehat{\rmT})$ 
  satisfies Condition~\ref{it:cond1} 
  since columns of the matrix $\calM(\widehat{\rmT})$ are indexed by 
  the spin-assignment of $\topedge{\widehat{\rmT}}=\topedge{\rmT}$
  and the rows of matrix $\calR_{n-2}$ by the spin-assignment of 
  $\bottom{\rmT}$.

We still need to show that 
  $\calR_{n-2} \cdot \calM(\widehat{\rmT})$ satisfies 
  Condition~\ref{it:cond2}.
By inductive hypothesis, we have that 
  $\calM(\widehat{\rmT})[\chi,\psi]$ is the number of satisfying states 
  of $\widehat{\rmT}$ if the spin-assignments  
  for $\bottom{\widehat{\rmT}}$ and 
  $\topedge{\widehat{\rmT}}$ are  $\chi$ and $\psi$, respectively.
By definition, $\calR_{n-2}[\phi,\chi]$ may be $1$ or $0$
  depending on whether or not the application of
  $\rmR_{n-2}$ to $(\widehat{\rmT},\bottom{\widehat{\rmT}})$ creates 
  a triangle for which a satisfying state is obtained by setting the
  spin-assignments of $\bottom{\rmT}$ 
  and of $\bottom{\widehat{\rmT}}$ equal to $\phi$ and $\chi$, respectively.
Therefore, $\calR_{n-2}[\phi,\chi]=1$ if and only if
  each satisfying state in $\widehat{\rmT}$ with spin-assignment 
  $\chi$ and $\psi$ for 
  $\bottom{\widehat{\rmT}}$ and $\topedge{\widehat{\rmT}}$ respectively,
  is a satisfying state in $\rmT$ with 
  spin-assignment $\phi$ and $\psi$ for $\bottom{\rmT}$ and 
  $\topedge{\rmT}$ respectively.
By definition of $\calM(\rmT)$, it immediately follows that
\begin{equation*}
\calM(\rmT)[\phi,\psi] = \sum_{\chi\in\{\texttt{+},\texttt{-}\}^2}
  \calR_{n-2}[\phi,\chi]\cdot\calM(\widehat{\rmT})[\chi,\psi] =
  \left(\calR_{n-2} \cdot \calM(\widehat{\rmT})\right)[\phi,\psi]\,,
\end{equation*}
and that $\calM(\rmT) = \calR_{n-2}\cdot\calM(\widehat{\rmT})$, thus
  concluding the inductive proof.
\end{proof}

\subsection{Satisfying states of strips of triangles }\label{subsub}
We now apply the transfer matrix method to count the number of
  satisfying states in any triangulation of a convex $n$-gon 
  $\rmT$ satisfying the condition $|\inner{\rmT}| = 0$.  
First, we observe that the matrices $\calW$ and $\calZ$ associated 
  to operations $\rmW$ and $\rmZ$, respectively, are given by:
\begin{displaymath}
\calW = 
\left(\begin{array}{cccc}
 0 & 0 &1 &0 \\
0 & 1 & 0& 1\\
1& 0& 1&0  \\
0&1 &0 &0
\end{array}\right)\,, \qquad\qquad\qquad
\calZ =
\left(\begin{array}{cccc}
 0 & 1 &0 &0 \\
1 & 1 & 0& 0\\
0& 0& 1&1  \\
0&0 &1 &0
\end{array}\right)\,.
\end{displaymath}
Note that $\calW=\Pi\cdot\calZ\cdot\Pi$ where $\Pi$ is the following 
  permutation matrix:
\begin{equation*}
\Pi = \left(\begin{array}{cccc}
1 & 0 & 0 & 0 \\
0 & 0 & 1 & 0  \\
0 & 1 & 0 & 0  \\
0 & 0 & 0 & 1
\end{array}\right)\,.
\end{equation*}
Since $\Pi^{-1}=\Pi$, for any $k \geq 0$ we get that
\begin{equation}\label{elemental}
\calW^k \ = \ (\Pi\cdot\calZ\cdot\Pi)^{k} 
        \ = \ \Pi\cdot\calZ^{k}\cdot\Pi\,. 
\end{equation}

\begin{theorem}\label{caso2} 
Let $\rmT_0$ be a degenerate triangulation,  $w_1,\ldots,w_m,z_1,\ldots,z_m$
  be a sequence of non--negative integers adding up to $n-2$ such that 
  $w_j \geq 1$ for $j\neq 1$ and
  $z_j \geq 1$ for $j\neq m$.
If $\rmT=\rmZ^{z_m}\circ \rmW^{w_m}\circ\ldots\circ
  \rmZ^{z_1}\circ\rmW^{w_1}(\rmT_0)$ and  $\calM= \calM(\rmT)$, then 
\begin{equation*}
\calM = \calZ^{z_m}\cdot\Pi\cdot\calZ^{w_m}\cdot\Pi\cdots
  \Pi\cdot\calZ^{z_1}\cdot\Pi\cdot\calZ^{w_1}\cdot\Pi\,.
\end{equation*}
Moreover, if $F_{k}$ denotes the $k$-th Fibonacci number, then
\begin{equation*}
\calM\cdot\one = \left(\begin{array}{c}
F_{n-1}  \\
F_{n}   \\
F_{n}  \\
F_{n-1} 
\end{array}\right)\,.
\end{equation*}
\end{theorem}
\begin{proof}
{From} Proposition~\ref{transfer} we have 
\begin{equation*}
\calM =  
  \calZ^{z_m}\cdot\calW^{w_m}\cdot\calZ^{z_{m-1}}\cdot\calW^{w_{m-1}} 
  \cdots\calZ^{z_{2}}\cdot\calW^{w_{2}}\cdot\calZ^{z_1}\cdot\calW^{w_1}\,.
\end{equation*}
By~(\ref{elemental}), the first stated identity immediately follows.

Now, for the second part, let $k\geq 1$.
Observe that 
\begin{eqnarray*}
\left(\begin{array}{cccc}
 0 & 1 \\ 1 & 1 
\end{array}\right)^k =
\left(\begin{array}{cccc}
 F_{k-1} & F_{k} \\
 F_{k} & F_{k+1} 
\end{array}\right)\,.
\end{eqnarray*}
It follows that,
\begin{equation}\label{zn} 
\calZ^k\cdot\one \ = \
\left(\begin{array}{cccc}
 F_{k-1} & F_{k}   & 0       & 0 \\
 F_{k}   & F_{k+1} & 0       & 0\\
 0       & 0       & F_{k+1} & F_{k}  \\
 0       & 0       & F_{k}   & F_{k-1}
\end{array}\right)\cdot\one 
\ = \
\left(\begin{array}{c}
F_{k+1}  \\
F_{k+2}  \\
F_{k+2}  \\
F_{k+1} 
\end{array}\right)\,.
\end{equation}
The first stated identity, the fact that
  $\Pi\cdot\calZ^{k}\cdot\one = \calZ^{k}\cdot\one$, 
  and observing that $\Pi\cdot\one=\one$, we get that
\begin{eqnarray*}
\calM\cdot \one & = & 
  \calZ^{z_m}\cdot\Pi\cdot\calZ^{w_m}\cdot\Pi\cdots\calZ^{z_1}\cdot\Pi\cdot\calZ^{w_1}\cdot\Pi\cdot\one \\
& = &
  \calZ^{z_m}\cdot\calZ^{w_m}\cdots\calZ^{z_1}\cdot\calZ^{w_1}\cdot\one\,.
\end{eqnarray*}
Since $\sum_{i=1}^{m}(z_i + w_i)=n-2$, the desired conclusion
  follows from~(\ref{zn}).
\end{proof}

\begin{prooff}{Theorem~\ref{result1}}
By hypothesis and Lemma~\ref{strip} we have that 
  for some degenerate triangulation $\rmT_0$ there are 
  non--negative integers $w_1,\ldots,w_m,z_1,\ldots,z_m$ adding up to
  $n-2$ such that $w_j \geq 1$ if $j\neq 1$,
  $z_j \geq 1$ if $j\neq m$, and
\begin{eqnarray*}
\rmT & = & \rmZ^{z_m}\circ\rmW^{w_m}\circ\ldots\circ
  \rmZ^{z_2}\circ\rmW^{w_2}\circ\rmZ^{z_1}\circ\rmW^{w_1}(\rmT_0)\,.
\end{eqnarray*} 
By Theorem~\ref{caso2}, we get that 
  $\degen{\rmT}=\one^{t}\cdot\calM(\rmT)\cdot\one = 2(F_{n}+F_{n-1})=2F_{n+1}$.
\end{prooff}

\medskip
We now obtain some intermediate results that we will need to prove
  Theorem~\ref{result2}:
Let $\rmT \in \Delta(\rmC_{n})$ and $\{\beta_1, \beta_2\}$ be an edge
  belonging to the boundary of $\rmT$. 
The \emph{satisfying vector} of $\rmT$ associated to node pair 
  $(\beta_1, \beta_2)$ denoted by $\vv_{\,\rmT}((\beta_1,\beta_2))$ 
  is a vector indexed by the spin-assignments 
  $\{\texttt{+},\texttt{-}\}^{2}$ of $(\beta_1, \beta_2)$, 
  so that $\vv_{\,\rmT}((\beta_1,\beta_2))[\psi]$ is equal
  to the number of satisfying states of~$\rmT$ if the spin-assignment of
  $(\beta_1, \beta_2)$ is equal to $\psi$.
For instance, by Theorem~\ref{caso2}, for every triangulation~$\rmT$
  of a convex $n$-gon with no interior triangles,
\begin{equation*}
\vv_{\,\rmT}(\bottom{\rmT}) = \left(\begin{array}{c}
F_{n-1}  \\
F_{n}   \\
F_{n}  \\
F_{n-1} 
\end{array}\right)\,.
\end{equation*}
Clearly, for every $\rmT \in \Delta(\rmC_{n})$ we have that 
\begin{equation}\label{igual} 
\vv_{\,\rmT}[\texttt{++}] = \vv_{\,\rmT} [\texttt{--}]\,, 
\qquad\qquad
\vv_{\,\rmT} [\texttt{+-}] = \vv_{\,\rmT}[\texttt{-+}]\,.
\end{equation}
Note that for edges $(\beta_1, \beta_2) \neq (\widehat{\beta}_1, \widehat{\beta}_2)$ 
  belonging to the boundary of $\rmT$, if
\begin{equation*}
\vv_{\,\rmT}((\beta_1, \beta_2)) = \left(\begin{array}{c}
\mathrm{x}  \\
\mathrm{y}   \\
\mathrm{y}  \\
\mathrm{x}
\end{array}\right)\,,
\qquad\qquad
\vv_{\,\rmT}((\widehat{\beta_1}, \widehat{\beta}_2)) = \left(\begin{array}{c}
\widehat{\mathrm{x}}  \\
\widehat{\mathrm{y}}   \\
\widehat{\mathrm{y}}  \\
\widehat{\mathrm{x}} 
\end{array}\right)\,,
\end{equation*}
then $2(\mathrm{x+y}) = 2(\widehat{\mathrm{x}} + \widehat{\mathrm{y}})$, or
  equivalently $\mathrm{x+y} = \widehat{\mathrm{x}} + \widehat{\mathrm{y}}$.

\begin{proposition}\label{parallel} If $\rmR \in \{\rmW, \rmZ\}$, 
  $\widehat{\rmT}\in\Delta(\rmC_{\widehat{n}})$, and 
  $\rmT = \rmR(\widehat{\rmT})$, then
\begin{equation*}
\vv_{\,\rmT}(\bottom{\rmT}) = 
  \calR \cdot \vv_{\,\widehat{\rmT}}(\bottom{\widehat{\rmT}})\,.
\end{equation*}
\end{proposition}
\begin{proof}
Implicit in the proof of Proposition~\ref{transfer}.
\end{proof}

We now define a useful operation on satisfying vectors.
Let $\bullet$ be the binary operator over $\NN^4$ defined by
\begin{eqnarray*}
\left(\begin{array}{cccc}
x_1 \\
x_2 \\
x_3\\
x_4
\end{array}\right) 
\bullet
\left(\begin{array}{cccc}
 y_1 \\
y_2 \\
y_3\\
y_4
\end{array}\right)
& = &
\left(\begin{array}{cccc}
 x_2\, y_3 \\
x_1 y_2 + x_3 y_4\\
x_4 y_3 + x_3 y_1\\
x_3 y_2 
\end{array}\right)\,.
\end{eqnarray*}

\begin{proposition}\label{hola} 
Let $\rmT_1\in \Delta(\rmC_{n_1})$ and 
  $\rmT_2 \in \Delta(\rmC_{n_2})$ be such that
  $\bottom{\rmT_1}= (\beta_{1}^1, \beta_{2}^1)$ and 
  $\bottom{\rmT_2}= (\beta_{1}^2,\beta_{2}^2)$. 
Then, 
\begin{equation*}
\vv_{\,\rmT_1 \bullet \rmT_2}((\beta_{1}^1,\beta_{2}^2)) 
   = \vv_{\,\rmT_1}((\beta_{1}^1, \beta_{2}^1)) 
        \bullet \vv_{\,\rmT_2}((\beta_{1}^2, \beta_{2}^2))\,.
\end{equation*}
\end{proposition}
\begin{proof}
To simplify the notation we denote 
  $\vv_{\,\rmT_1\bullet\rmT_2}((\beta_{1}^1,\beta_{2}^2))$, 
  $\vv_{\,\rmT_1}((\beta_{1}^1, \beta_{2}^1))$
  and $\vv_{\,\rmT_2}((\beta_{1}^2,\beta_{2}^2))$
  by
  $\vv_{\beta_{1}^1\beta_{2}^2}$,
  $\vv_{\beta_{1}^1\beta_{2}^1}$, 
  and
  $\vv_{\beta_{1}^2 \beta_{2}^2}$, respectively.
For $i\in\{1,2\}$, we know that 
  $\vv_{\beta_{1}^i\beta_{2}^i}[\psi]$
  is equal to the number of
  satisfying states of $\rmT_i$ if 
  $\psi \in \{\texttt{+},\texttt{-}\}^2$ is
  the spin-assignment for $(\beta_{1}^i,\beta_{2}^i)$.
We consider the following cases depending on the spin-assignment
  of $(\beta_{1}^1, \beta_{2}^2)$.
\begin{itemize}
\item 
Spin-assignment of $(\beta_{1}^1, \beta_{2}^2)$ is $\texttt{++}$:
Since $\texttt{+++}$ is not a satisfying assignment for the
  triangle $(\beta_{1}^1,\beta,\beta_{2}^2)$ of $\rmT$,
  if the spin-assignment of $\beta = \beta_{1}^2=\beta_{2}^1$ is $\texttt{+}$, 
  then the state of $\rmT$ is not satisfying.
If the spin assignment of $(\beta_{1}^1,\beta, \beta_{2}^2)$ is
  $\texttt{+}\texttt{-}\texttt{+}$, each satisfying state of $\rmT_1$
  and $\rmT_2$ (with spin-assignment for $(\beta_{1}^1, \beta_{2}^1)$
  equal to $\texttt{+}\texttt{-}$ and spin-assignment for 
  $(\beta_{1}^2,\beta_{2}^2)$ equal to $\texttt{-}\texttt{+}$) 
  is a satisfying state for $\rmT$, and 
\begin{equation*}
\vv_{\beta_{1}^1\beta_{2}^2}[\texttt{++}]=
   \vv_{\beta_{1}^1\beta_{2}^1}[\texttt{+-}]\cdot
   \vv_{\beta_{1}^2\beta_{2}^2}[\texttt{-+}]\,.
\end{equation*}

\item 
Spin-assignment of $(\beta_{1}^1, \beta_{2}^2)$ is $\texttt{+}\texttt{-}$:
Note that 
  the triangle $(\beta_{1}^1,\beta, \beta_{2}^2)$ with spin-assignment 
  $\texttt{+}\texttt{+}\texttt{-}$ fulfills the condition
  of satisfying state. 
Hence, each satisfying state of $\rmT_1$ and $\rmT_2$ 
  (with spin-assignment for $(\beta_{1}^1, \beta_{2}^1)$ equal
  to $\texttt{++}$ and spin-assignment for $(\beta_{1}^2, \beta_{2}^2)$
  equal to $\texttt{+}\texttt{-}$) is a satisfying state for $\rmT$.
Analogously, if the spin-assignment of $\beta$ is equal to $\texttt{-}$,
  each satisfying state of $\rmT_1$ and $\rmT_2$ (with spin-assignment
  for $(\beta_{1}^1, \beta_{2}^1)$ equal to $\texttt{+}\texttt{-}$ and
  spin-assignment for $(\beta_{1}^2,\beta_{2}^2)$ equal to
  $\texttt{--}$) is a satisfying state for $\rmT$. 
It follows that
\begin{equation*}
\vv_{\beta_{1}^1\beta_{2}^2}[\texttt{+-}] = 
  \vv_{\beta_{1}^1 \beta_{2}^1}[\texttt{++}]\vv_{\beta_{1}^2\beta_{2}^2}[\texttt{+-}]  + \vv_{\beta_{1}^1\beta_{2}^1}[\texttt{+-}] \vv_{\beta_{1}^2\beta_{2}^2}[\texttt{--}]\,.
\end{equation*}
\end{itemize}
By a symmetry argument, we also have that
\begin{eqnarray*}
\vv_{\beta_{1}^1\beta_{2}^2}[\texttt{--}] & = &
  \vv_{\beta_{1}^1\beta_{2}^1}[\texttt{-+}]\vv_{\beta_{1}^2\beta_{2}^2}[\texttt{+-}]\,, \\
\vv_{\beta_{1}^1\beta_{2}^2}[\texttt{-+}] & = & 
  \vv_{\beta_{1}^1\beta_{2}^1}[\texttt{-+}]\vv_{\beta_{1}^2\beta_{2}^2}[\texttt{++}] + \vv_{\beta_{1}^1\beta_{2}^1}[\texttt{--}]\vv_{\beta_{1}^2\beta_{2}^2}[\texttt{-+}]\,.
\end{eqnarray*}
\end{proof}



We now recall some basic well known facts about Fibonacci numbers.
Let $\varphi$ denote the golden ration.
If $F_n$ denotes the $n$-th  Fibonacci number, it is well known
  that $F_{n+1}=F_{n}+F_{n-2}$ for all $n\geq 1$, and that 
\begin{equation*}
F_n \ = \ \frac{\varphi^n -(-\frac{1}{\varphi})^n}{\sqrt{5}}
\end{equation*}
It immediately follows that for all $n\geq 1$,
\begin{equation}\label{bounds}
\varphi^{n-2} 
  \ \leq\ F_n 
  \ \leq \frac{1 + \left(\frac{1}{\varphi}\right)^{2}}{\sqrt{5}} \varphi^n \ \leq\ \varphi^n \,.
\end{equation}

\begin{lemma}\label{o2}
If $\rmT$ is a triangulation of a convex $n$-gon, then
  $\varphi^{n -|\inner{\rmT}|} 
   \geq 
  \varphi^{2}(\sqrt{\varphi})^n$.
\end{lemma}
\begin{proof}
Since $|\outter{\rmT}|\geq |\inner{\rmT}|+2$ and
  $|\outter{\rmT}|+|\inner{\rmT}|=n-2$, we get that
  $|\inner{\rmT}| \leq (n/2) -2$.
The claimed result immediately follows.
\end{proof}

\begin{prooff}{of Theorem~\ref{result2}}
We claim that for any triangulation of a convex $n$-gon $\rmT$ such 
  that $|\inner{\rmT}| = m$ it holds that 
  $\degen{\rmT}\geq \varphi^{n-m}$.
To prove this claim we proceed by induction on~$m$.
If $m=0$, by Theorem~\ref{result1} we have that 
  $\degen{\rmT} = 2 F_{n+1}$. 
Using the lower bound in~(\ref{bounds}) we obtain 
  $\degen{\rmT} \geq 2  \varphi^{n-1} \geq \varphi^n$.
If $m=1$, by Lemma~\ref{lemma2} we know that
  for some $n_1, n_2, n_3\geq 2$ such that $n_1+n_2+n_3=n$
  there are triangulations $\rmT_1 \in \Delta(\rmC_{n_1+1})$ and
  $\rmT_2 \in \Delta(\rmC_{n_2+1})$ such that
  $|\inner{\rmT_1}|=|\inner{\rmT_2}|=0$,  
  and basic operations 
  $\rmR_{1},\ldots,\rmR_{n_3-1}\in\{\rmW,\rmZ\}$ such that
\begin{equation*}
\rmT = \rmR_{n_3-1}\circ\cdots\circ\rmR_{2}\circ\rmR_{1}
  (\rmT_1 \bullet \rmT_2)\,.
\end{equation*}
By Theorem~\ref{caso2}, for $i\in\{1,2\}$ we know that
\begin{equation*}
\vv_{\,\rmT_i}(\bottom{\rmT_i}) = 
  \left(\begin{array}{cccc}
  F_{n_i} \\
  F_{n_i+1} \\
  F_{n_i+1}\\
  F_{n_i}
  \end{array}\right)\,.
\end{equation*}
Now, denote $\vv_{\,\rmT_1\bullet\rmT_2}(\bottom{\rmT_1\bullet\rmT_2})$ 
  by $\vv$.
Observe that Proposition~\ref{hola} and the definition of $\bullet$
  imply that
\begin{equation*}
\vv
   = \left(\begin{array}{cccc}
     F_{n_1} \\
     F_{n_1+1} \\
     F_{n_1+1}\\
     F_{n_1}
     \end{array}\right) \bullet 
     \left(\begin{array}{cccc}
     F_{n_2} \\
     F_{n_2+1} \\
     F_{n_2+1}\\
     F_{n_2}
     \end{array}\right)= 
     \left(\begin{array}{cccc}
     F_{n_1+1}F_{n_2+1} \\
     F_{n_1}F_{n_2+1} + F_{n_1+1}F_{n_2} \\
     F_{n_1}F_{n_2+1} + F_{n_1+1}F_{n_2} \\
     F_{n_1+1}F_{n_2+1}
\end{array}\right)\,. 
\end{equation*}
Repeated application of Proposition~\ref{parallel} yields that
\begin{equation*}
\degen{\rmT} = 
  \one^{t}\cdot \calR_{n_3-1}\cdots\calR_{2}\cdot\calR_{1}\cdot\vv\,.
\end{equation*}
By~(\ref{igual}) and due to the block structure of $\calZ$,
  we have that $\Pi\cdot\vv= \vv$ and 
  $\Pi\cdot\calZ^{q}\cdot\vv = \calZ^{q}\cdot\vv$, for every $q \geq 0$.  
Therefore, 
  since $\calW=\Pi\cdot\calZ\cdot\Pi$, 
  the last displayed identity may be rewritten as 
  $\degen{\rmT} = \one^{t}\cdot\calZ^{n_3-1}\cdot\vv$.
Hence,
\begin{eqnarray*}
\degen{\rmT} & = & \one^{t}\cdot
  \left(\begin{array}{cccc}
  F_{n_3 -2} & F_{n_3-1} & 0 & 0 \\
  F_{n_3-1} & F_{n_3} &0 & 0 \\
  0 & 0 & F_{n_3} & F_{n_3-1} \\
  0 & 0 & F_{n_3-1} & F_{n_3 -2}
  \end{array}\right)\cdot\vv  \\
  & = & 2\left(F_{n_3}F_{n_1+1}F_{n_2+1}
         +      F_{n_3+1}(F_{n_1}F_{n_2+1} + F_{n_1+1}F_{n_2})\right)\,.
\end{eqnarray*}
Since Fibonacci numbers satisfy the identity
  $F_{p+q}=F_{p}F_{q-1}+F_{p+1}F_{q}$, we get that
\begin{eqnarray*}
\degen{\rmT} 
  & = &  
  2\left(F_{n_3}(F_{n_1+1}F_{n_2+1}+F_{n_1}F_{n_2+1}+F_{n_1+1}F_{n_2}) 
    +F_{n_3-1}(F_{n_1}F_{n_2+1}+F_{n_1+1}F_{n_2})\right) \\
& = &
  2\left(F_{n_3}(F_{n_1+2}F_{n_2+1}+F_{n_1+1}F_{n_2}) 
    +F_{n_3-1}(F_{n_1+2}F_{n_2}+F_{n_1+1}F_{n_2-1}-F_{n_1-1}F_{n_2-1})\right) \\
& = & 
  2(F_{n_3}F_{n_1+n_2+2}+F_{n_3-1}(F_{n_1+n_2+1}-F_{n_1-1}F_{n_2-1}))\\
& = & 
  2(F_{n_1+n_2+n_3+1} - F_{n_3-1}F_{n_1-1}F_{n_2-1})\,.
\end{eqnarray*}
Since $n =n_1 + n_2 + n_3$, $2>\varphi$ and $\varphi^{2}-1 = \varphi$, 
  by~(\ref{bounds}) it follows that
\begin{equation*}
  \degen{\rmT} 
\ \geq\
  2\varphi^{n_1+n_2+n_3-1}\left(1-\varphi^{-2}\right)  
\ \geq\
  \varphi^{n-1}\,.
\end{equation*}
Now, suppose the claim holds for every triangulation 
  $\rmT\in\Delta(\rmC_{n})$ such that $|\inner{\rmT}|< m$. 
Let $\rmT\in\Delta(\rmC_{n})$ be such that $|\inner{\rmT}|=m$. 

We know from Lemma~\ref{mainlemma} that
  there is a $\widetilde{\rmT}\in\Delta(\rmC_{\widetilde{n}})$ 
  such that $|\inner{\widetilde{\rmT}}|=0$, 
  a $\widehat{\rmT}\in\Delta(\rmC_{\widehat{n}})$ satisfying 
  condition~\ref{it:lemtwo} of Lemma~\ref{mainlemma},
  basic operations 
  $\rmR_1,\ldots,\rmR_{l}\in\{\rmW,\rmZ\}$ where $l\geq 1$, and
  $n = \widehat{n}+ \widetilde{n} + l -1$ such that
\begin{equation*}
\rmT = \rmR_{l}\circ \cdots \circ\rmR_{1}(\widetilde{\rmT}\bullet\widehat{\rmT})\,.
\end{equation*}
By an argument similar to the one used to handle the $m=1$ case, 
  we have that
\begin{equation*}
\degen{\rmT} = \one^{t}\cdot\calR_{l}\cdots\calR_{2}\cdot\calR_{1}\cdot
  \vv_{\,\widetilde{\rmT}\bullet\widehat{\rmT}}(\bottom{\widetilde{\rmT}\bullet\widehat{\rmT}})\,.
\end{equation*}
Since $\widetilde{\rmT}\in\Delta(\rmC_{\widetilde{n}})$ is such that 
  $|\inner{\widetilde{\rmT}}|=0$, by  Theorem~\ref{caso2} we have that
\begin{equation*}
\vv_{\,\widetilde{\rmT}}(\bottom{\widetilde{\rmT}}) = 
\left(\begin{array}{cccc}
  F_{\widetilde{n}-1} \\
  F_{\widetilde{n}} \\
  F_{\widetilde{n}} \\
  F_{\widetilde{n}-1}
\end{array}\right)\,. 
\end{equation*}
Let $\widehat{x}$ and $\widehat{y}$
  denote $\vv_{\,\widehat{\rmT}}(\bottom{\widehat{\rmT}})[\texttt{++}]$
  and $\vv_{\,\widehat{\rmT}}(\bottom{\widehat{\rmT}})[\texttt{+-}]$ respectively.
Observe that~(\ref{igual}) implies that
  $\vv_{\,\widehat{\rmT}}(\bottom{\widehat{\rmT}})[\texttt{-+}]=\widehat{y}$
  and $\vv_{\,\widehat{\rmT}}(\bottom{\widehat{\rmT}})[\texttt{--}]=\widehat{x}$.
Hence, by Proposition~\ref{hola},
\begin{equation*}
\vv_{\,\widetilde{\rmT}\bullet\widehat{\rmT}}(\bottom{\widetilde{\rmT}\bullet\widehat{\rmT}}) 
  = \left(\begin{array}{cccc}
  F_{\widetilde{n}-1} \\
  F_{\widetilde{n}} \\
  F_{\widetilde{n}}\\
  F_{\widetilde{n}-1}
  \end{array}\right)\bullet
  \left(\begin{array}{cccc}
  \widehat{x} \\
  \widehat{y} \\
  \widehat{y}\\
  \widehat{x}
  \end{array}\right) =
  \left(
\begin{array}{cccc}
 \widehat{y}F_{\widetilde{n}} \\
 \widehat{x}F_{\widetilde{n}}+\widehat{y}F_{\widetilde{n}-1} \\
 \widehat{x}F_{\widetilde{n}}+\widehat{y}F_{\widetilde{n}-1} \\
 \widehat{y}F_{\widetilde{n}}
\end{array}\right)\,.
\end{equation*}
Denoting 
  $\vv = \vv_{\,\widetilde{\rmT}\bullet\widehat{\rmT}}(\bottom{\widetilde{\rmT}\bullet\widehat{T}})$ 
  we again observe that~(\ref{igual}) implies 
  that $\Pi\cdot\vv= \vv$ and $\Pi\cdot\calZ^{q}\cdot\vv =\calZ^{q}\cdot\vv$
  for all $q\geq 0$. 
Putting everything together we conclude that
\begin{eqnarray*}
\degen{\rmT} & = & \one^{t}\cdot\calZ^{l}\cdot
\left(
\begin{array}{cccc}
 \widehat{y}F_{\widetilde{n}} \\
 \widehat{x}F_{\widetilde{n}}+\widehat{y}F_{\widetilde{n}-1} \\
 \widehat{x}F_{\widetilde{n}}+\widehat{y}F_{\widetilde{n}-1} \\
 \widehat{y}F_{\widetilde{n}}
\end{array}\right) \\
& = &
  2(\widehat{x}F_{l+2}F_{\widetilde{n}}+\widehat{y}(F_{l+1}F_{\widetilde{n}}
    +F_{l+2}F_{\widetilde{n}-1}))\,.
\end{eqnarray*}
The lower bound for Fibonacci numbers given in~(\ref{bounds})
  and the fact that $2>\varphi$
  imply that
\begin{eqnarray*}
\degen{\rmT} 
& \geq & 
2\left(\widehat{x}\varphi^{l+\widetilde{n}-2}+2\widehat{y}\varphi^{l+\widetilde{n}-3}\right) \\
& \geq & 2(\widehat{x}+\widehat{y})\varphi^{l+\widetilde{n}-2}\,. 
\end{eqnarray*}
Recalling that $\degen{\widehat{\rmT}} = 2(\widehat{x}+\widehat{y})$
  and observing that conditions 1 and 2 of Lemma~\ref{mainlemma} 
  guarantee that  $|\inner{\widehat{\rmT}}|$ is equal to $m-1$ or  
  $m-2$, from the inductive hypothesis we obtain that
  $\degen{\widehat{\rmT}}\geq \varphi^{\widehat{n}-(m-1)}$.
It follows that
  $\degen{\rmT} \geq 
    \varphi^{\widehat{n}+\widetilde{n}+l- 2 - (m-1)}
    = \varphi^{n-m}$.
This concludes the inductive prove of the claim.
Lemma~\ref{o2} immediately implies the desired result.
\end{prooff}

\section{Conclusion}\label{sec:conclusion}
We have established that the number of satisfying states of any triangulation of 
  a convex $n$-gon es exponential in $n$.
It would be of interest to generalize this result to more general
  triangulations.
Two natural cases to address next are triangulations that are embedable 
  over low genus surfaces and $k$-trees.


\bibliographystyle{alpha}
\bibliography{biblio}

\end{document}